\documentclass{amsart}
\usepackage{amssymb}

\usepackage[utf8]{inputenc}

\usepackage{amsmath}
\usepackage{amssymb}
\usepackage{array}

\usepackage{verbatim}

\usepackage{alltt}

\usepackage{fancybox}
\def\R{{\mathbb R}}
\def\C{{\mathbb C}}
\def\N{{\mathbb N}}
\def\Z{{\mathbb Z}}

\def\<{\langle}
\def\>{\rangle}
\def\P{\mathbb P}

\def\E{\mathbb E}

\def\0{\underline 0}

\def\1{\underline 1}

\newcommand{\bel}{\begin{equation}\label}
\newcommand{\nobel}{\begin{equation}}
\newcommand{\ee}{\end{equation}}
      \newtheorem{theorem}{Theorem}[section]
       \newtheorem{proposition}[theorem]{Proposition}
       \newtheorem{corollary}[theorem]{Corollary}
       
       \newtheorem{remark}{Remark}[section]
\theoremstyle{definition}
\newtheorem{definition}{Definition}[section]

\title{Independence preserving property of Kummer laws}
\author{Angelo Efo\'evi Koudou and Jacek Weso\l owski}
\address{}
\date{\today}

\begin{document}

\begin{abstract}
We prove that if $X,Y$ are positive, independent, non-Dirac random variables and if for $\alpha,\beta\ge 0$, $\alpha\neq \beta$,  
$$
\psi_{\alpha,\beta}(x,y)=\left(y\,\tfrac{1+\beta(x+y)}{1+\alpha x+\beta y},\;x\,\tfrac{1+\alpha(x+y)}{1+\alpha x+\beta y}\right),
$$ then the random variables $U$ and $V$ defined by $(U,V)=\psi_{\alpha,\beta}(X,Y)$ 
are independent if and only if $X$ and $Y$ follow Kummer distributions with suitably related parameters. In other words,  any invariant measure for a lattice recursion model governed by $\psi_{\alpha,\beta}$ in the scheme introduced by Croydon and Sasada in \cite{CS2020}  is necessarily a product measure with Kummer marginals. The result extends earlier characterizations of Kummer and gamma laws by independence of
$$
U=\tfrac{Y}{1+X}\quad\mbox{and}\quad V= X\left(1+\tfrac{Y}{1+X}\right),
$$
which corresponds to the case of $\psi_{1,0}$.

We also show that this independence property of Kummer laws covers, as limiting cases, several independence models known in the literature: the Lukacs, the Kummer-Gamma, the Matsumoto-Yor and the discrete Korteweg de Vries models. \end{abstract}

\maketitle

\section{Introduction}

Consider, for  $b,c>0$, the gamma distribution $\mathrm{Gamma}(b,c)$ with density proportional to
$$y^{b-1}e^{-cy}I_{(0,\infty)}(y)
$$
and for $p \in \R, a > 0, b > 0$, the generalized inverse Gaussian (GIG) distribution GIG$(p,a,b)$ with density proportional to
		$$x^{p-1}e^{-ax-b/x} I_{(0,\infty)}(x).$$
Following \cite{CS2020}, we say that a  quadruplet of probability measures $(\mu, \nu, \tilde{\mu}, \tilde{\nu})$ on $\mathcal{U}, \mathcal{V}, \tilde{\mathcal{U}}, \tilde{\mathcal{V}}$, respectively, 
satisfy the {\it detailed balance equation} for a map $F: \mathcal{U} \times \mathcal{V} \rightarrow  \tilde{\mathcal{U}} \times \tilde{\mathcal{V}}$ if 
$$F(\mu \otimes \nu)=\tilde{\mu} \otimes \tilde{\nu},$$
where $F(\mu \otimes \nu)$ means $(\mu \otimes \nu) \circ F^{-1}$.

The Matsumoto-Yor property is the following: for $p, a , b >0$,  given two independent, positive random variables $X$ and $Y$ such that $X \sim {\rm GIG}(-p, a,b)$ and $Y \sim \mathrm{Gamma}(p,a)$,  the random variables $\frac{1}{X+Y}$ and $\frac{1}{X}-\frac{1}{X+Y}$ are independent (and follow the ${\rm GIG}(-p, b,a)$ and $\mathrm{Gamma}(p,b)$, respectively). Using the terminology of \cite{CS2020}, the Matsumoto-Yor property says that 
the quadruplet of probability measures $\mu={\rm GIG}(-p, a,b),\,\nu= \mathrm{Gamma}(p,a),\, \tilde{\mu}= {\rm GIG}(-p, b,a),\,\tilde{\nu}=  \mathrm{Gamma}(p,b)$ satisfy the {\it detailed balance equation} for the map 
$$F: \,   (0,\infty )^2 \rightarrow  (0,\infty )^2 $$
$$(x,y)\stackrel{F}{\mapsto} \left(\frac{1}{x+y}, \frac{1}{x}-\frac{1}{x+y}\right).$$
This property was discovered in  \cite{MY01} in the case $a=b$, while studying some functionals of exponential Brownian motion. In \cite{LW00} it was noticed that it   holds also when $a\neq b$  and it was proved that this independence property is in fact a characterization: for  two non-Dirac, positive and independent random variables $X$ and $Y$,
the random variables  $\frac{1}{X+Y}$ and $\frac{1}{X}-\frac{1}{X+Y}$ are independent if and only if
 $X \sim {\rm GIG}(-p, a,b)$ while  $Y \sim \mathrm{Gamma}(p,a)$ for some  $p, a,b >0$.

In \cite{KV12} the authors studied the question of finding decreasing and bijective functions $f:  (0,\infty ) \rightarrow (0,\infty )$ such that there exists a  quadruplet of probability measures $(\mu, \nu, \tilde{\mu}, \tilde{\nu})$ on  $(0,\infty )$ satisfying the  detailed balance equation for the map  
$$T_f: \,   (0,\infty )^2 \rightarrow  (0,\infty )^2 $$
$$\label{transformation} (x,y) \mapsto  (f(x+y), \ f(x)-f(x+y)).$$  This led, at the cost of some regularity assumptions, to other independence properties of Matsumoto-Yor type 
(of course, one retrieves the original Matsumoto-Yor case for $f(x)=1/x$), amongst which 
was a property involving the Kummer distribution: for $a,c>0$ and $b\in\R$, the Kummer distribution $\mathcal K(a,b,c)$ is defined through the density proportional to
$$ \tfrac{x^{a-1}e^{-cx}}{(1+x)^b}\,I_{(0,\infty)}(x).
$$ 
More precisely, it was proved in \cite{KV12} that if $X$ and $Y$ are independent Kummer and gamma random variables with suitably related parameters then
$$
U=X+Y\quad \mbox{and}\quad V=\tfrac{1+(X+Y)^{-1}}{1+X^{-1}}
$$ 
are independent Kummer and beta random variables. This was the starting point of a number of works on Matsumoto-Yor type characterizations of the Kummer distribution. Firstly, 
starting from the latter property and looking for an involutive version of it, i.e. trying to find an involutive map $F: \,(X,Y)\mapsto (U,V)$ for which the Kummer distribution is involved in a detailed balance equation, the following interesting property was discovered in 
\cite{HV16}:
let $X$ and $Y$ be independent, $X$ have the distribution $\mathcal{K}(a,b,c)$ 
and $Y$ have the gamma distribution $\mathrm{Gamma}(b,c)$,
 then 
\bel{def1}U=\tfrac{Y}{1+X}
\qquad \mbox{and}\qquad V=X\,\tfrac{1+X+Y}{1+X} \ee
  are independent, $U\sim \mathcal{K}(b,a,c)$ and $V\sim\mathrm{Gamma}(a,c)$. 

In  \cite{PW18} this independence property was proved to give a characterization result with no assumption of existence of densities.
Related characterizations were considered in \cite{Wes15} and \cite{PW16}. In \cite{KP20} an extension to  the matrix-variate case was established, while in \cite{Pil21} a free probability version of the property and characterization was given. The latter needed a definition of a new distribution, a free analogue of the Kummer distribution. 

In the past ten years there has been a revival of interest in transformations preserving independence properties, triggered mostly by invariance properties of random iterations schemes or random walks in random environment. In particular, the log-gamma random polymer introduced  in \cite{Sep12} relies on the Lukacs independence property for the Gamma distributions (see \cite{L55}), while the beta polymer introduced in \cite{BC17} refers to an independence property for beta distributions, called neutrality. This property can be traced back to  the characterization of the bivariate Dirichlet distribution of \cite{DR71}  (see also \cite{SW03}). For a complete description of all stationary polymers on the lattice $\mathbb Z^2_+$ see \cite{CN18}, where also many further references can be found.

Let us come back to the definition of the detailed balance equation and recall its context as described in \cite{CS2020},
which considers models assuming the following dynamics: for $(n,t)$ in $\Z ^2$, $n$ is the spatial coordinate and $t$  the 
temporal one.  For fixed $t\in \Z$, $(x_n^t)_{n\in \Z } \in (0,\infty)^\Z $ is the configuration of the system at time $t$, and $(y_n^t)_{n\in \Z } \in (0,\infty)^\Z$ a collection of auxiliary variables through which the dynamics from $t$ to $t+1$ are defined. Namely, $(x_n^t, y_n^t)$ depends  on $(x_n^{t-1}, y_{n-1}^t)$ only, 
\begin{equation}\label{Zrec}
	(x_n^t, y_n^t)=G(x_n^{t-1}, y_{n-1}^t),
\end{equation} 
where for a bijection $F:\mathcal X\times\mathcal Y\to\tilde{\mathcal X}\times\tilde{\mathcal Y}$ either $G=F$, when $n+t$ is even or $G=F^{-1}$ when $n+t$ is odd. The case when $F$ is involutive is referred to as type I model, while the general case is referred to as type II model. Let $x=(x_n)_{n\in\Z}$ be such that the above recursion with the initial condition $x^0_n=x_n$, $n\in \Z$, has a unique solution $(x_n^t(x),y_n^t(x))_{n,t\in\Z}$. Let $\mathcal X^*$ denote the set of all such $x$'s. According to Theorem 1.1 in \cite{CS2020}, for a type I model, a sequence of iid random variables $X=(X_n)_{n\in \Z}$ with $X_1\sim \mu$ satisfies $X\stackrel{d}{=}\left(x^1_n(X)\right)_{n\in\Z}$ iff there exists a probability measure $\nu$ such that the pair $(\mu,\nu)$ satisfies the detailed balance condition with respect to $F$. That is, $\mu\otimes\nu$ is the invariant measure for the recursion \eqref{Zrec}. In case of the type II model similar alternating invariance holds for pairs $\mu\otimes\nu$ and $\tilde{\mu}\otimes\tilde{\nu}$ depending on parity of $n+t$ in \eqref{Zrec}. In \cite{CS2020} the authors identified four such type I and/or type II models.
\begin{enumerate}
	\item Ultra-discrete KdV (Korteweg-de Vries): type I model with 
	$$
	F(x,y):=F_{udK}^{(J,K)}=(y-(x+y-J)_++(x+y-K)_+,\;x-(x+y-K)_++(x+y-J)_+)
	$$ with $\mu$ and $\nu$ the shifted truncated exponential or shifted scaled truncated geometric laws.
	\item  Discrete KdV: type I model with
	$$
	F(x,y):=F_{dK}^{(\alpha,\beta)}(x,y)=\left(\tfrac{y(1+\beta xy)}{1+\alpha xy},\,\tfrac{x(1+\alpha xy)}{1+\beta xy}\right)
	$$ with $\mu$ the GIG law and $\nu$ the GIG (gamma) law which, when $\alpha\beta=0$, has a direct connection with the Matsumoto-Yor property and related characterization of the GIG and gamma laws.
	\item  Ultra-discrete Toda: type II model with
	$$
	F(x,y):=F_{udT^*}(x\wedge y,\,x-y)
	$$  with $\mu$, $\nu$, $\tilde{\mu}$ the shifted exponential, $\tilde{\nu}$ asymmetric Laplace or $\mu$, $\nu$, $\tilde{\mu}$ shifted scaled geometric,  $\tilde{\nu}$ scaled discrete Laplace laws; this one is related to classical characterizations of the exponential and geometric distributions from \cite{Fer64} and \cite{Cra66}.
	\item  Discrete Toda: type II model with
	$$
	F(x,y):=F_{dT^*}(x,y)=\left(x+y,\tfrac{x}{x+y}\right)
	$$
	with $\mu$, $\nu$, $\tilde{\mu}$ the gamma, $\tilde{\nu}$ the beta laws having a direct connection with the characterization of the gamma distribution given in \cite{L55}.
\end{enumerate}
For relations to box-ball systems and  Pitman's transform one can consult \cite{CST22} and \cite{CKST23}. Also recently a  mysterious connection between  independence properties and Yang-Baxter equations holding for related transformations was discovered in \cite{SU22}.

In the context of the discrete KdV model in 
 \cite{CS2020} the authors observed that if  $X$ and $Y$ are independent,  $U$ and $V$ are independent and all four have GIG distributions with suitable parameters, then $(X,Y)$ and $(U,V)$ satisfy the detailed balance equation for the map  $F_{dK}^{(\alpha, \beta)}$. They also conjectured that the GIG distributions are the only possible ones that satisfy the $F_{dK}^{(\alpha, \beta)}$-detailed balance equation. Recently, in \cite{LW22} this conjecture was proved without the assumptions of existence and regularity of densities made in \cite{BN21} in proving a version of the same conjecture. More precisely, \cite{LW22}  established the following extension of the  Matsumoto-Yor property: if $A$ and $B$ are non-degenerate, positive and independent random variables, and if $\alpha$ and $\beta$ are two positive and distinct numbers, then the random variables 
$$S=\tfrac{1}{B} \tfrac{\beta A + B}{\alpha A + B}, \ \ \ T=\tfrac{1}{A} \tfrac{\beta A + B}{\alpha A + B}$$
are independent if and only if $A$ and $B$ have GIG distributions with suitable parameters. 

In this paper we reveal one more candidate for an invariant measure for a lattice recursion model. We derive the detailed balance equation for the Kummer distributions.  
Specifically, our main result (Theorem \ref{gHVcharacterization}) gives a characterization of the Kummer laws, which is a result of a  similar nature as the one in  \cite{LW22} for the GIG laws, i.e. it says that the Kummer distributions are the only possible ones which let the detailed balance equation be satisfied for the map
\bel{eff} F(x,y)=\left(y\,\tfrac{1+\beta(x+y)}{1+\alpha x+\beta y}, \, x\,\tfrac{1+\alpha(x+y)}{1+\alpha x+\beta y} \right).\ee 
The proof uses a suitably designed "Laplace-type" transform and leads to a special second order linear differential equation for an unknown function of such form. In this sense the general methodology (a Laplace type transform and a second order ordinary linear differential equation) resembles  that of the proof from  \cite{LW22}. However, at the technical level, the challenges to overcome were of quite a different nature.  
Interpreting this result in the context of the lattice system of recursions described above, it says that the Kummer distributions are the only relevant invariant measures for
the type I model governed by the $F$ defined in \eqref{eff}.

The paper is organized as follows: in Section 2 we introduce a scaled version of the Kummer distribution, we express and prove the considered independence property
in terms of that scaled Kummer distribution (Theorem \ref{gHV}).  
We show in Section 3 that some limit versions of Theorem \ref{gHV} cover several well-known independence preserving transformations.
Indeed, relying on a version of Theorem 5.5 from \cite{Bil68}, we  prove that Theorem 2.1 yields, as limiting cases,  the Lukacs property, \cite{L55}, the Kummer-Gamma property, \cite{KV12}, the Matsumoto-Yor property, \cite{MY01, LW00} and the KdV property, \cite{CS2020}. 
In Section 4 we define and analyze the {\it Kummer transform}, an extended Laplace transform that will help us to prove the chacterization theorem 
formulated at the end of this section. Section 5 contains the proof of the characterization split in several steps (subsections), amongst which is the crucial observation that the unknown {\em Kummer transform}  satisfies the Kummer differential equation (see \cite{AbraSt}).

\section{The independence property}

For the purpose of this paper it will be convenient to introduce a scaled version of the Kummer distribution.
\begin{definition}
	Let $\mathcal K_\alpha(a,b,c)$ for $\alpha\ge 0$, $a,c>0$ and $b\in\R$ be the probability distribution defined by the density
	$$
	f(x)\propto \tfrac{x^{a-1}e^{-cx}}{(1+\alpha x)^b}\,I_{(0,\infty)}(x)
	$$
\end{definition}

\begin{theorem}
	\label{gHV}
	Assume that $$(X,Y)\sim\mathcal K_{\alpha}(a,b,c)\otimes \mathcal K_{\beta}(b,a,c)$$ for $a,b,c>0$ and $\alpha,\beta\ge 0$, $\alpha\neq \beta$. Let 
	\begin{equation}\label{UV}
	U=Y\,\tfrac{1+\beta(X+Y)}{1+\alpha X+\beta Y}\quad\mbox{and}\quad V=X\,\tfrac{1+\alpha(X+Y)}{1+\alpha X+\beta Y}.
	\end{equation}
	Then
	$$
	(U,V)\sim \mathcal K_{\alpha}(b,a,c)\otimes \mathcal K_{\beta}(a,b,c).
	$$
\end{theorem}
\begin{remark}\label{exten}
	Note that  $\mathcal K_0(a,b,c)=\mathrm{Gamma}(a,c)$ and for $\alpha>0$ and $X\sim \mathcal K_{\alpha}(a,b,c)$ we have $\alpha X\sim \mathcal K(a,b,c/\alpha)$.
Consequently, by taking $(\alpha,\beta)=(1,0)$, we see that Theorem \ref{gHV} is a straightforward extension of the independence property observed in \cite{HV16} and recalled in the introduction, see \eqref{def1}.
\end{remark}
\begin{proof}[Proof of Theorem \ref{gHV}]
Denote 
$$
\psi(x,y)=\left(y\,\tfrac{1+\beta(x+y)}{1+\alpha x+\beta y},\,x\,\tfrac{1+\alpha(x+y)}{1+\alpha x+\beta y}\right)=:(u,v),\qquad x,y>0.
$$	
Note that $\psi:(0,\infty)^2\to(0,\infty)^2$ is an involution. Moreover,  the following identities hold true:
\begin{align}
\label{sum} & x+y=u+v, \\
\label{ratxv} & \tfrac{x}{1+\beta y}=\tfrac{v}{1+\alpha u},\\
\label{ratyu} & \tfrac{y}{1+\alpha x}=\tfrac{u}{1+\beta v}.
\end{align}
	
Using \eqref{sum}, \eqref{ratxv} and \eqref{ratyu} we compute the Jacobian $J_{\psi^{-1}}(u,v)$ of $\psi^{-1}=\psi$ as follows
\bel{jac}
J_{\psi^{-1}}(u,v)=\left|\tfrac{\partial(x,y)}{\partial(u,v)}\right|=\left|\begin{array}{cc}\tfrac{\partial x}{\partial u} & \tfrac{\partial x}{\partial v}\\ 
1-\tfrac{\partial x}{\partial u} & 1-\tfrac{\partial x}{\partial v}\end{array}\right|=\tfrac{\partial x}{\partial u}-\tfrac{\partial x}{\partial v}=\tfrac{1+\alpha x+\beta y}{1+\alpha u+\beta v}=\tfrac{xy}{uv}.
\ee


Now we are ready to find the joint density of $(U,V)$.  We have
\begin{align*}
f_{(U,V)}(u,v)&=\left|J_{\psi^{-1}}(u,v)\right|\,f_X(x(u,v))\,f_Y(y(u,v))\propto 
\tfrac{xy}{uv}\,\tfrac{x^{a-1}}{(1+\alpha x)^b}e^{-cx}\,\tfrac{y^{b-1}}{(1+\beta y)^a}e^{-cy}I_{(0,\infty)^2}(u,v)\\
&=\tfrac{1}{uv}\,\left(\tfrac{x}{1+\beta y}\right)^a\,\left(\tfrac{y}{1+\alpha x}\right)^b\,e^{-c(x+y)}\,I_{(0,\infty)^2}(u,v).
\end{align*}
Again referring to \eqref{sum}, \eqref{ratxv} and \eqref{ratyu} we get
$$
f_{(U,V)}(u,v)\propto \tfrac{1}{uv}\left(\tfrac{v}{1+\alpha u}\right)^a\,\left(\tfrac{u}{1+\beta v}\right)^b\,e^{-c(u+v)}\,I_{(0,\infty)^2}(u,v)
$$
and the result follows.
\end{proof}

Theorem \ref{gHV} yields the following independence property for pure (i.e. $\alpha=\beta=1$) Kummer variables.
	
	\begin{corollary}
		\label{sK}
		Let $(\tilde X,\tilde Y)\sim\,\mathcal K(a,b,c)\otimes\mathcal K(b,a,\gamma c)$ for $1\neq \gamma>0$. Then
		$$
		(\tilde U,\tilde V):=\left(\tilde Y\,\tfrac{\tilde X+\gamma(1+\tilde Y)}{1+\tilde X+\tilde Y},\;\tilde X\tfrac{\tilde Y+\gamma^{-1}(1+\tilde X)}{1+\tilde X+\tilde Y}\right)\sim \mathcal K(b,a,c)\otimes\mathcal K(a,b,\gamma c).
		$$
	\end{corollary}
\begin{proof} For $\alpha,\beta\neq 0$  denote  $\tilde X=\alpha X$, $\tilde Y=\beta Y$, $\tilde U=\alpha U$, $\tilde V=\beta V$. Let $\gamma=\alpha/\beta$. Then $\alpha\neq \beta$ and, in view of Remark \ref{exten}, reference to Theorem \ref{gHV} ends the proof.
\end{proof}	 

In a very interesting recent paper, \cite{SU22}, the authors announced (independently) a  result (see their Theorem 1.1 (ii)), which is equivalent to the independence property observed in  Theorem \ref{gHV}. 

\section{Limiting cases of the Kummer independence property} 
We will show that Theorem \ref{gHV} yields, as limiting cases, several well-known independence properties: the Lukacs property, \cite{L55}, the Kummer-Gamma property, \cite{KV12}, the Matsumoto-Yor property, \cite{MY01, LW00} and the KdV property, \cite{CS2020}.

We will rely on the following version of Theorem 5.5 from \cite{Bil68}.

\begin{theorem}\label{Bil}
	Let $X_n\stackrel{d}{\to} X$, with $X_n$ and $X$ assuming values in a separable metric space $S$. Let $\phi_n,\phi:S\to S$ be measurable functions such that for any $x\in S$ and any sequence $x_n\to x$ we have $\phi_n(x_n)\to\phi(x)$. 
	
	Then
	$$
	\phi_n(X_n)\stackrel{d}{\to}\phi(X).
	$$
\end{theorem}

Note that except the equality  $\mathcal K_0(a,b,c)=\mathrm{Gamma}(a,c)$, in case $b>a>0$ we have $\mathcal K_1(a,b,0)=\mathrm{Beta}_{II}(a,b-a)$, where  the $\mathrm{Beta}_{II}(p,q)$ law with $p,q>0$ is defined by the density
	$$
	f(x)\propto \tfrac{x^{p-1}}{(1+x)^{p+q}}\,\mathbf 1_{(0,\infty)}.
	$$

Several other distributions can be obtained by taking weak limits of Kummer laws.


\begin{proposition}\label{gamma1} If $\alpha\to\infty$ then
	
	\begin{enumerate} 
		\item when $a>b$ 
	\begin{equation}\label{ga}
		\mathcal K_\alpha(a,b,c)\stackrel{w}{\to}\mathrm{Gamma}(a-b,c),
		\end{equation}
	\item when $b>a$ 
		\begin{equation}\label{be}
			\mathcal K_1(a,b,c/\alpha)\stackrel{w}{\to}
			\mathrm{Beta}_{II}(a,b-a),\end{equation}
\item when $b>0$ 
	\begin{equation}\label{GIG}
		\mathcal K_\alpha(\alpha b+a_1,\alpha b+a_2,c)\stackrel{w}{\to}\mathrm{GIG}(a_1-a_2,c,b).
	\end{equation}
\item when $b>0$
\begin{equation}\label{GIG'}
\mathcal K_\alpha(a \sqrt{\alpha}+b,a\sqrt{\alpha},c)\stackrel{w}{\to}\mathrm{Gamma}(b,c).
\end{equation}
\item when $b>0$ 
\begin{equation}\label{iga}
\mathcal K_\alpha(a \alpha,a\alpha+b,c/\alpha)\stackrel{w}{\to}\mathrm{InvGamma}(b,a),
\end{equation}where $\mathrm{InvGamma}(b,c)$ is defined by the density
$$
f(x)\propto x^{-b-1}\,e^{-c/x}\,\mathbf 1_{(0,\infty)}(x).
$$
\end{enumerate}
\end{proposition}	

\subsection{The Lukacs property}
The following independence property  is related to the  characterization of the gamma laws proved in \cite{L55}.

\begin{theorem} Assume that $$(X_1,\,Y_1)\sim \mathrm{Gamma}(a_1,c)\otimes\mathrm{Gamma}(b_1,c).$$
	
	Let
	$$
	(U_1,V_1)=\left(\tfrac{Y_1}{X_1},\,X_1+Y_1\right).
	$$
	
	Then $$(U_1,V_1)\sim\mathrm{Beta}_{II}(b_1,a_1)\otimes\mathrm{Gamma}(a_1+b_1,c).$$
\end{theorem}

\begin{proof}
	In view of Theorem \ref{gHV} with $\alpha=n$, $\beta=0$  and Remark \ref{exten} we see that 
	$$(X_1^{(n)},Y_1)\sim \mathcal K_n(a_1+b_1,b_1,c)\otimes \mathrm{Gamma}(b_1,c)$$
	implies that for 
	$$
	\phi_n(x,y)=\left(\tfrac{n y}{1+n x},\,x\tfrac{1+n(x+y)}{1+n x}\right)
	$$
	we have
	\begin{equation}\label{L2}
		\phi_n(X_1^{(n)},Y_1)\sim \mathcal K_1(b_1,a_1+b_1,c/n)\otimes\mathrm{Gamma}(a_1+b_1,c).
	\end{equation}
	
	Since, see  \eqref{ga},  
	$$
	\mathcal K_n(a_1+b_1,b_1,c)\stackrel{w}{\to}\mathrm{Gamma}(a_1,c)\quad \mbox{as }\;n\to\infty,
	$$
	we see that 
	\begin{equation}\label{L1}(X_1^{(n)},Y_1)\stackrel{d}{\to}(X_1,Y_1)\sim \mathrm{Gamma}(a_1,c)\otimes \mathrm{Gamma}(b_1,c).
	\end{equation}
	
	Moreover, for $x_n\to x>0$ and $y>0$ 
	$$
\phi_n(x_n,y)=	\left(\tfrac{y}{x_n+1/n},\,\tfrac{x_n}{x_n+1/n}\,(x_n+y+1/n)\right)\to\left(\tfrac{y}{x},\,x+y\right)=:\phi(x,y).
	$$
	Thus, in view of \eqref{L1}, Theorem \ref{Bil} implies
	\begin{equation}\label{L3}
		\phi_n(X_1^{(n)},Y_1)\stackrel{d}{\to}\phi(X_1,Y_1)=\left(\tfrac{Y_1}{X_1},\,X_1+Y_1\right).
	\end{equation}
	
	On the other hand, \eqref{L2}, in view of \eqref{be}, yields  
	\begin{equation}\label{L4}
		\mathbb P_{\phi_n(X_1^{(n)},Y_1)}\stackrel{w}{\to} \mathrm{Beta}_{II}(b_1,a_1)\otimes \mathrm{Gamma}(a_1+b_1,c).
	\end{equation}
	
	Comparing \eqref{L3} and \eqref{L4} we conclude that
	$$
	(U_1,V_1)=\left(\tfrac{Y_1}{X_1},\,X_1+Y_1\right)\sim \mathrm{Beta}_{II}(b_1,a_1)\otimes \mathrm{Gamma}(a_1+b_1,c).
	$$
\end{proof}

\subsection{The Kummer-Gamma independence property}
The following result was proved in \cite{KV12}; see also \cite{Wes15}.
\begin{theorem}
	Assume that 
	$$
	(X_2,Y_2)\sim \mathcal K(a_2,a_2+b_2,c_2)\otimes \mathrm{Gamma}(b_2,c_2).
	$$
	
	Let
	$$
	(U_2,V_2)=\left(X_2+Y_2,\;\tfrac{1+\tfrac{1}{X_2+Y_2}}{1+\tfrac{1}{X_2}}\right).
	$$
	
	Then
	$$
	(U_2,V_2)\sim \mathcal K(a_2+b_2,a_2,c_2)\otimes\mathrm{Beta}_I(a_2,b_2),
	$$
	where $\mathrm{Beta}_I(p,q)$ has the density
	$$f(y)\propto y^{p-1}(1-y)^{q-1}\mathbf 1_{(0,1)}(y).$$
\end{theorem}
\begin{proof}
	In view of Theorem \ref{gHV} with $\alpha=1$ and $\beta=n$ we see that 
	$$
	(X_2,\,Y_2^{(n)})\sim\mathcal K_1(a_2,a_2+b_2,c_2)\otimes \mathcal K_n(a_2+b_2,a_2,c_2)
	$$
	implies that for
	$$
	\widetilde{\phi}_n(x,y)=\left(y\,\tfrac{1+n(x+y)}{1+x+n y},\,n x\,\tfrac{1+x+y}{1+x+n y}\right)
	$$
	we have
	\begin{equation}\label{KV2}
		\widetilde\phi_n(X_2,Y_2^{(n)})\sim \mathcal K_1(a_2+b_2,a_2,c_2)\otimes \mathcal K_1(a_2, a_2+b_2,c_2/n).
	\end{equation}

	Since, see  \eqref{ga},
	$$
	\mathcal K_n(a_2+b_2, a_2,c)\stackrel{w}{\to}\mathrm{Gamma}(b_2,c_2),
	$$
	we see that
	\begin{equation}\label{KV1}
		\mathbb P_{X_2,\,Y_2^{(n)}}\stackrel{w}{\to} \mathcal K_1(a_2,a_2+b_2,c_2)\otimes \mathrm{Gamma}(b_2,c_2).
	\end{equation}
	Moreover, for any $x>0$ and $y_n\to y>0$ we have
	$$
	\widetilde{\phi}_n(x,y_n)=\left(\tfrac{y_n}{y_n+(1+x)/n}(x+y_n+\tfrac1{n}),\tfrac{x}{y_n+(1+x)/n}(1+x+y_n)\right)
	\to \left(x+y,\,\tfrac{x}{y}(1+x+y)\right)=:\widetilde{\phi}(x,y).
$$
	
	Consequently, by Theorem \ref{Bil}, in view of \eqref{KV1},  we have 
	\begin{equation}\label{KV4}
		\widetilde \phi_n(X_2,Y_2^{(n)})\stackrel{d}{\to}\widetilde \phi(X_2,Y_2)=\left(X_2+Y_2,\,\tfrac{X_2}{Y_2}(1+X_2+Y_2)\right).
	\end{equation}
	
	On the other hand, \eqref{KV2}, in view of  \eqref{be}, yields 
	\begin{equation}\label{KV5}
		\mathbb P_{\widetilde{\phi}_n(X_2,Y_2^{(n)})}\stackrel{w}{\to} \mathcal K_1(a_2+b_2,a_2,c)\otimes\mathrm{Beta}_{II}(a_2,b_2).
	\end{equation}
	
	Comparing \eqref{KV4} and \eqref{KV5} we conclude that
	$$
	(U_2,\widetilde V_2):=\left(X_2+Y_2,\,\tfrac{X_2}{Y_2}(1+X_2+Y_2)\right)\sim \mathcal K_1(a_2+b_2,a_2,c)\otimes\mathrm{Beta}_{II}(a_2,b_2).
	$$
	
	But for $h(u,v)=(u,\tfrac{v}{1+v})$ and $\phi=h\circ \widetilde \phi$ we have
	$$
	\phi(x,y)=\left(x+y,\,\tfrac{\tfrac{x}{y}(1+x+y)}{1+\tfrac{x}{y}(1+x+y)}\right)=\left(x+y,\,\tfrac{1+\tfrac{1}{x+y}}{1+\tfrac{1}{x}}\right).
	$$
	Thus
	$$(U_2,\,V_2)=\left(X_2+Y_2,\,\tfrac{1+\tfrac{1}{X_2+Y_2}}{1+\tfrac{1}{X_2}}\right)=\phi(X_2,Y_2)=\left(U_2,\tfrac{\widetilde V_2}{1+\widetilde V_2}\right)\sim \mathcal K_1(a_2+b_2,a_2,c)\otimes \mathrm{Beta}_{I}(a_2,b_2).$$
	Here we used the fact that if $Z\sim \mathrm{Beta}_{II}(p,q)$ then $\tfrac{Z}{1+Z}\sim \mathrm{Beta}_I(p,q)$.

\end{proof}

\subsection{The Matsumoto-Yor property}
The following result was proved in Matsumoto and Yor \cite{MY01, MY03}; see also  \cite{LW00}.
\begin{theorem}Assume that
	$$
	(X_3,Y_3)\sim \mathrm{GIG}(-a_3,b_3,c_3)\otimes \mathrm{Gamma}(a_3,b_3).
	$$
	Let
	$$
	(U_3,V_3)=\left(\tfrac{1}{X_3+Y_3},\,\tfrac{1}{X_3}-\tfrac{1}{X_3+Y_3}\right).
	$$
	Then
	$$
	(U_3,V_3)\sim \mathrm{GIG}(-a_3,c_3,b_3)\otimes \mathrm{Gamma}(a_3,c_3).
	$$
\end{theorem}
\begin{proof}
	In view of Theorem \ref{gHV} with $\alpha=n$ and $\beta=n^2$ we see that if 
	\begin{equation}\label{MY1}\left(X_3^{(n)},\,Y_3^{(n)}\right)\sim\mathcal K_{n}(nc_3, nc_3+a_3,b_3)\otimes \mathcal K_{n^2}(nc_3+a_3,nc_3, b_3)\end{equation}
	then with 
	$$
	\widetilde\phi_n(x,y)=\left(y\tfrac{1+n^2(x+y)}{1+nx+n^2y},\,nx\tfrac{1+n(x+y)}{1+nx+n^2y}\right)
	$$
	we have
	\begin{equation}\label{MY2}
		\widetilde\phi_n\left(X_3^{(n)},\,Y_3^{(n)}\right)\sim \mathcal K_n(nc_3+a_3,nc_3,b_3)\otimes  \mathcal K_{n}(nc_3,nc_3+a_3,b_3/n).
	\end{equation}
	Since, see \eqref{GIG} and \eqref{GIG'},
		$$
	K_n(nc_3,nc_3+a_3,b_3)\stackrel{w}{\to}\mathrm{GIG}(-a_3,b_3,c_3)
	$$
	and
	$$
	\mathcal K_{n^2}(nc_3+a_3,nc_3,b_3)\stackrel{w}{\to}\mathrm{Gamma}(a_3,b_3)
	$$
	we see that
	\begin{equation}\label{MY3}
		\left(X_3^{(n)},\,Y_3^{(n)}\right)\stackrel{d}{\to}(X_3,Y_3)\sim \mathrm{GIG}(-a_3,b_3,c_3)\otimes\mathrm{Gamma}(a_3,b_3).
	\end{equation}
	
	Moreover, if $x_n\to x>0$ and $y_n\to y>0$ then we have
	$$
	\widetilde \phi_n(x_n,y_n)=\left(\tfrac{y}{y+\frac{x}{n}+\frac{1}{n^2}}(x+y+\tfrac{1}{n^2}),\,x\tfrac{x+y+\frac{1}{n}}{y+\frac{x}{n}+\frac{1}{n^2}}\right)\to\left(x+y,\,\tfrac{x(x+y)}{y}\right).
	$$
	
	Consequently, by \eqref{MY3}, in view of Theorem \ref{Bil} we have 
	\begin{equation}\label{MY4}
		\widetilde\phi_n\left(X_3^{(n)},\,Y_3^{(n)}\right)\stackrel{d}{\to}\widetilde \phi(X_3,Y_3)=\left(X_3+Y_3,\tfrac{X_3(X_3+Y_3)}{Y_3}\right).
	\end{equation}
	
	On the other hand, since \eqref{GIG} and \eqref{iga},
	$$
	K_n(nc_3+a_3,nc_3,b_3)\stackrel{w}{\to} \mathrm{GIG}(a_3,b_3,c_3)
	$$
	and
	$$
	\mathcal K_{n}(nc_3,nc_3+a_3,b_3/n)\stackrel{w}{\to} \mathrm{InvGamma}(a_3,c_3)
	$$
	we conclude that
	\begin{equation}\label{MY5}
		\mathbb P_{\widetilde\phi_n\left(X_3^{(n)},\,Y_3^{(n)}\right)}\stackrel{w}{\to}\mathrm{GIG}(a_3,b_3,c_3)\otimes  \mathrm{InvGamma}(a_3,c_3).
	\end{equation}
	Combining \eqref{MY4} with \eqref{MY5} we get
	$$
	(\widetilde U_3,\,\widetilde V_3):=\left(X_3+Y_3,\,\tfrac{X_3(X_3+Y_3)}{Y_3}\right)\sim \mathrm{GIG}(a_3,b_3,c_3)\otimes  \mathrm{InvGamma}(a_3,c_3).
	$$
	
	Denoting $h(u,v)=(u^{-1},v^{-1})$ and $\phi=h\circ \widetilde\phi$, we get
	$$
	\phi(x,y)=\left(\tfrac{1}{x+y},\,\tfrac{y}{x(x+y)}\right)=\left(\tfrac{1}{x+y},\,\tfrac{1}{x}-\tfrac{1}{x+y}\right).
	$$
	
	Finally, we obtain
	$$
	(U_3,V_3)=\left(\tfrac{1}{X_3+Y_3},\,\tfrac{1}{X_3}-\tfrac{1}{X_3+Y_3}\right)=\phi(X_3,Y_3)=\left(\tfrac{1}{\widetilde U_3},\,\tfrac{1}{\widetilde V_3}\right)\sim\mathrm{GIG}(-a_3,c_3,b_3)\otimes\mathrm{Gamma}(a_3,c_3)
	$$
\end{proof}

\subsection{The discrete KdV independence property}
The following result was proved in \cite{CS2020}; see also \cite{BN21} and \cite{LW22}.
\begin{theorem}Assume that
	\begin{equation}\label{GMY0}
		(X_4,Y_4)\sim \mathrm{GIG}(-a_4,\alpha b_4,c_4)\otimes \mathrm{GIG}(-a_4,\beta c_4,b_4).
	\end{equation}
	Let
	$$
	(U_4,V_4)=\left(Y_4\tfrac{1+\alpha X_4Y_4}{1+\beta X_4Y_4},\,X_4\tfrac{1+\beta X_4Y_4}{1+\alpha X_4Y_4}\right).
	$$
	Then
	\begin{equation}\label{GMY00}
		(U_4,V_4)\sim \mathrm{GIG}(-a_4,\alpha c_4,b_4)\otimes \mathrm{GIG}(-a_4,\beta b_4,c_4).
	\end{equation}
\end{theorem}
\begin{proof} 
	In view of Theorem \ref{gHV} with $\alpha$ changed into $n/\alpha$ and $\beta$ changed into $n/\beta$ we see that if 
	\begin{equation}\label{GMY1}\left(X_4^{(n)},\,Y_4^{(n)}\right)\sim\mathcal K_{n/\alpha}(nb_4+a_4,nb_4,c_4)\otimes \mathcal K_{n/\beta}(nb_4,nb_4+a_4,c_4)\end{equation}
	then with 
	$$
	\widetilde\phi_n(x,y)=\left(y\,\tfrac{1+\tfrac{n}{\beta} (x+y)}{1+n\left(\tfrac{x}{\alpha}+\tfrac{y}{\beta}\right)},\,x\,\tfrac{1+\tfrac{n}{\alpha}(x+y)}{1+n\left(\tfrac{x}{\alpha}+\tfrac{y}{\beta}\right)}\right)
	$$
	we have
	\begin{equation}\label{GMY2}
		\widetilde\phi_n\left(X_4^{(n)},\,Y_4^{(n)}\right)\sim \mathcal K_{n/\alpha}(nb_4,nb_4+a_4,c_4)\otimes  \mathcal K_{n/\beta}(nb_4+a_4,nb_4,c_4).
	\end{equation}
	Since, see \eqref{GIG}, 
	$$
	K_{n/\alpha}(nb_4+a_4,nb_4,c_4)\stackrel{w}{\to} \mathrm{GIG}(a_4,c_4,\alpha b_4)
	$$
	and
	$$
	K_{n/\beta}(nb_4,nb_4+a_4,c_4)\stackrel{w}{\to} \mathrm{GIG}(-a_4,c_4,\beta b_4),
	$$ in view of \eqref{GMY1}, we see that
	\begin{equation}\label{GMY3}
		\left( X_4^{(n)},\,Y_4^{(n)}\right)\stackrel{d}{\to}(\widetilde X_4,\,\widetilde Y_4)\sim\mathrm{GIG}(a_4,c_4,\alpha b_4)\otimes \mathrm{GIG}(-a_4,c_4,\beta b_4).
	\end{equation}
	
	Moreover, if $x_n\to x>0$ and $y_n\to y>0$ then 
	$$
	\widetilde \phi_n(x_n,y_n)=\left(\alpha y\tfrac{x+y+\frac{\beta}{n}}{\beta x+\alpha y+\frac{\alpha\beta}{n}},\,\beta x\tfrac{x+y+\frac{\alpha}{n}}{\beta x+\alpha y+\frac{\alpha\beta}{n}}\right)\to\left(\alpha y\,\tfrac{x+y}{\beta x+\alpha y},\,\beta x\tfrac{x+y}{\beta x+\alpha y}\right)=:\widetilde\phi(x,y).
	$$
	
	Consequently, by \eqref{GMY3}, in view of Theorem \ref{Bil} we have 
	\begin{equation}\label{GMY4}
		\widetilde\phi_n\left(X_4^{(n)},\,Y_4^{(n)}\right)\stackrel{d}{\to}\widetilde \phi(\widetilde X_4,\widetilde Y_4)=\left(\alpha \widetilde Y_4\,\tfrac{\widetilde X_4+\widetilde Y_4}{\beta \widetilde X_4+\alpha \widetilde Y_4},\,\beta \widetilde X_4\tfrac{\widetilde X_4+\widetilde Y_4}{\beta \widetilde X_4+\alpha \widetilde Y_4}\right).
	\end{equation}
	
	On the other hand since, see again \eqref{GIG},
	$$
	\mathcal K_{n/\alpha}(nb_4,nb_4+a_4,c_4)\stackrel{w}{\to}\mathrm{GIG}(-a_4,c_4,\alpha b_4)
	$$
	and 
	$$
	K_{n/\beta}(nb_4+a_4,nb_4,c_4)\stackrel{w}{\to}\mathrm{GIG}(a_4,c_4,\beta b_4)
	$$
	by \eqref{GMY2} we conclude that
	\begin{equation}\label{GMY5}
		\mathbb P_{\widetilde \phi_n(X_4^{(n)},Y_4^{(n)})}\stackrel{w}{\to}\mathrm{GIG}(-a_4,c_4,\alpha b_4)\otimes \mathrm{GIG}(a_4,c_4,\beta b_4).
	\end{equation}
	
	Combining \eqref{GMY4} and \eqref{GMY5} we get
	\begin{equation}\label{GMY6}
		(\widetilde U_4,\widetilde V_4):=\widetilde \phi(\widetilde X_4,\widetilde Y_4)\sim \mathrm{GIG}(-a_4,c_4,\alpha b_4)\otimes \mathrm{GIG}(a_4,c_4,\beta b_4).
	\end{equation}
	
	Note that for $g(x,y)=(1/x,\beta y)$, in view of \eqref{GMY0} and \eqref{GMY3}, we have 
	$$
	g(X_4,Y_4)\stackrel{d}{=}(\widetilde X_4,\widetilde Y_4)
	$$
	(note that if $R \sim \mathrm{GIG}(p,a,b)$ anf if $k>0$, then $kR \sim \mathrm{GIG}(p,a/k,kb)$) and for $h(x,y)=(x/\alpha,1/y)$, in view of \eqref{GMY4} and \eqref{GMY00}, we see that
	$$
	h(\widetilde U_4,\widetilde V_4)\stackrel{d}{=}(U_4,V_4).
	$$
	
	Therefore, defining $\phi=h\circ\widetilde \phi\circ g$ we get  $(U_4,V_4)\stackrel{d}{=}\phi(X_4,Y_4)$. But
	$$
	\phi(x,y)=\left(y\tfrac{1+\alpha xy}{1+\beta xy},\,x\tfrac{1+\beta xy}{1+\alpha xy}\right),
	$$
	which concludes the proof.
	
\end{proof}

\section{The Kummer transform and the characterization}

For a positive random variable $W$ and $\gamma\ge 0$ consider an extended Laplace transform $L_W^{(\gamma)}$ of the form
$$
L_W^{(\gamma)}(s,t,z)=\E\,\tfrac{W^s}{(1+\gamma W)^t}\,e^{-zW}.
$$
We will call it the Kummer transform. Note that the Kummer tranform is well defined at least for $s,z>0$ and $t\in\R$. Moreover, for any fixed $s>0$, $t\in\R$, the Kummer transform as a function of $z>0$,  is just the Laplace transform of the measure $\tfrac{w^s}{(1+\gamma w)^t}\P_W(dw)$, so it uniquely determines the distribution of $W$.
Note also that
\bel{id1}
L_W^{(\gamma)}(s,t,z)+\gamma L_W^{(\gamma)}(s+1,t,z)=L_W^{(\gamma)}(s,t-1,z)
\ee
and for any $k=1,2,\ldots$
\bel{id2}
\tfrac{\partial ^k\,L_W^{(\gamma)}(s,t,z)}{\partial z^k}=(-1)^kL_W^{(\gamma)}(s+k,t,z).
\ee

\begin{proposition}\label{KLT}
	Let $X\sim \mathcal K_{\alpha}(a,b,c)$, $a,c>0$, $b\in\R$. Then
	\begin{equation}\label{KLTK}
	L_X^{(\alpha)}(s,t,z)=\tfrac{\Gamma(a+s)U\left(a+s,a+s-b-t+1,\tfrac{c+z}{\alpha}\right)}{\alpha^s\Gamma(a)U\left(a,a-b+1,\tfrac{c}{\alpha}\right)},
	\end{equation}
	$s>0$, $t\in\R$, $z>-c$,
	where $U$ is the Kummer function (see 13.2.5 in \cite{AbraSt}) defined by
	\begin{equation}\label{Kufu}
	U(a,b,z)=\tfrac{1}{\Gamma(a)}\,\int_0^{\infty}\,\tfrac{x^{a-1}}{(1+x)^{a-b+1}}\,e^{-zx}\,dx,\qquad a,z>0,\;b\in\R. 
	\end{equation}
\end{proposition}	

\begin{proof} It is a simple consequence of that fact that due to the definition of the Kummer function $U$ in \eqref{Kufu} the normalizing constant of the Kummer distribution $\mathcal K_\alpha(a,b,c)$ has the form 
$$
\tfrac{\alpha^a}{\Gamma(a)U\left(a,a-b+1,\tfrac{c}{\alpha}\right)}.
$$
\end{proof}

\begin{remark}\label{gamma}
	Note that when $b=a+1$, in view of \eqref{Kufu}, we have
	$$
	U(a,a+1,z)=z^{-a},
	$$
	whence, if $X\sim\mathcal K_\alpha(a,0,c)$ then \eqref{KLTK} gives
	$$
	L_X^{(\alpha)}(0,0,z)=\tfrac{c^a}{(c+z)^a},
	$$
	which implies that $X$ is a Gamma random variable, $\mathrm{Gamma}(a,c)$.
\end{remark}

\begin{proposition}\label{unique}Let $b\in\R$, $a,c,\alpha>0$. 
Assume that for some fixed  $(s,t)\in(0,\infty)\times\R$ and all $z>0$
\begin{equation}\label{KLtr}
L_X^{(\alpha)}(s,t,z)=k(s,t)U\left(a+s,a+s-b-t+1,\tfrac{c+z}{\alpha}\right),
\end{equation}
where $k(s,t)$ is a constant (depending also on $\alpha,a,b,c$).
Then $X\sim\mathcal K_{\alpha}(a,b,c)$.
\end{proposition}
\begin{proof}
Fix some $z_0>0$. Then \eqref{KLtr} implies
\begin{equation}\label{U/U}
\tfrac{L_X^{(\alpha)}(s,t,z+z_0)}{L_X^{(\alpha)}(s,t,z_0)}=\tfrac{U\left(a+s,a+s-b-t+1,\tfrac{c+z_0+z}{\alpha}\right)}{U\left(a+s,a+s-b-t+1,\tfrac{c+z_0}{\alpha}\right)},\quad z\ge 0.
\end{equation}
This is the Laplace transform of a random variable $Y$ with distribution
\begin{equation}\label{PYdy}
\P_Y(dy)=\tfrac{\tfrac{y^s}{(1+\alpha y)^t}\,e^{-z_0y}\,\P_X(dy)}{L_X^{(\alpha)}(s,t,z_0)}.\end{equation}
In view of Proposition \ref{KLT}, by \eqref{KLTK} and \eqref{U/U}, the random variable $Y$ has the Kummer distribution $\mathcal K_{\alpha}(a+s,b+t,c+z_0)$, i.e.
$$
\P_Y(dy)\propto \tfrac{y^{a+s-1}}{(1+\alpha y)^{b+t}}\,e^{-(c+z_0)y}\,dy.
$$ 
The result follows by comparing the last formula with \eqref{PYdy}.
\end{proof}	

\begin{remark}\label{equind}[Alternative proof of Theorem \ref{gHV}.]
	Let $X$ and $Y$ be independent and $U,V$ are defined as in \eqref{UV}. In view of \eqref{sum}, \eqref{ratxv} and \eqref{ratyu} we then see that if 
	\begin{equation}\label{KLeq}
	L_X^{(\alpha)}(s,t,z)\,L_Y^{(\beta)}(t,s,z)=L_U^{(\alpha)}(t,s,z)\,L_V^{(\beta)}(s,t,z),\quad (s,t,z)\in(0,\infty)\times\R\times(0,\infty)
	\end{equation}
    then $U$ and $V$ are independent. 
	
	Let $(X,Y)\sim\mathcal K_{\alpha}(a,b;c)\otimes \mathcal K_{\beta}(b,a;c)$. Consequently, to prove Theorem \ref{gHV}, it suffices to check that identity  \eqref{KLeq} is satisfied for $(U,V)\sim \mathcal K_{\alpha}(b,a;c)\otimes \mathcal K_{\beta}(a,b;c)$. Referring to \eqref{KLTK}, we see that \eqref{KLeq} is then  equivalent to 
	\begin{equation}\label{Uiden}
	\tfrac{U\left(a+s,a+s-b-t+1,\tfrac{c+z}{\alpha}\right)\,U\left(b+t,b+t-a-s+1,\tfrac{c+z}{\beta}\right)}{\alpha^s\beta^t\,U\left(a,a-b+1,\tfrac{c}{\alpha}\right)\,U\left(b,b-a+1,\tfrac{c}\beta\right)}=\tfrac{U\left(b+t,b+t-a-s+1,\tfrac{c+z}\alpha\right)\,U\left(a+s,a+s-b-t+1,\tfrac{c+z}\beta\right)}{\alpha^t\beta^s\,U\left(b,b-a+1,\tfrac{c}\alpha\right)\,U\left(a,a-b+1,\tfrac{c}\beta\right)}.
	\end{equation}
	To see that \eqref{Uiden} holds true we rely on the following identity for the Kummer function $U$ (see (13.1.29) in \cite{AbraSt})
	\begin{equation}\label{AS:Kum}
	U(a,b,z)=z^{1-b}U(1+a-b,2-b,z).
	\end{equation}
	
	Note also, that in view of \eqref{KLeq}  and \eqref{KLTK} of Proposition \ref{KLT} we have
	$$
	\tfrac{L_X^{(\alpha)}(s,t,z)}{L_U^{(\alpha)}(t,s,z)}=\tfrac{L_V^{(\beta)}(s,t,z)}{L_Y^{(\beta)}(t,s,z)}=c^{a-b}\tfrac{\Gamma(b)}{\Gamma(a)}\,\tfrac{\Gamma(a+s)}{\Gamma(b+t)}\,(c+z)^{b-a+t-s}.
	$$	
\end{remark}

Now we are ready to formulate the main result which is a characterization of Kummer laws by the detailed balance condition with respect to the function $F$ given in \eqref{eff}.
\begin{theorem}
	\label{gHVcharacterization}
Let  $\alpha,\beta\ge 0$, $\alpha\neq \beta$. Let $X,Y$ be positive, independent, non-Dirac random variables and define 
$$U=Y\,\tfrac{1+\beta(X+Y)}{1+\alpha X+\beta Y}\quad\mbox{and}\quad V=X\,\tfrac{1+\alpha(X+Y)}{1+\alpha X+\beta Y}.$$
If $U$ and $V$ are independent, then there exist $a,b,c>0$ such that 
$$(X,Y)\sim\mathcal K_{\alpha}(a,b;c)\otimes \mathcal K_{\beta}(b,a;c).$$
We then have
	$$
	(U,V)\sim \mathcal K_{\alpha}(b,a;c)\otimes \mathcal K_{\beta}(a,b;c).
	$$
\end{theorem}  
\begin{remark}
	Recall that Theorem 2.6 of \cite{PW18} says that for random variables $X$ and $Y$ which are independent, positive, non-Dirac, independence  of $U$ and $V$ given by \eqref{def1} implies that $X\sim\mathcal K(a,b,c)$ and $Y\sim\mathrm{Gamma}(b,c)$. (Note the change of parametrization of the Kummer distribution: instead of $b-a$ in \cite{PW18} we write here just $b$.) In view of the first part of Remark \ref{exten} this result covers the case $(\alpha,\beta)=(1,0)$ in Theorem \ref{gHVcharacterization}. Due to the second part of Remark \ref{exten} and symmetry with respect to $\alpha$ and $\beta$ we conclude that the cases $\beta=0$, $\alpha>0$ and $\beta>0$, $\alpha=0$ also follow immediately from Theorem 2.6. of \cite{PW18}. That is, we need only to prove Theorem \ref{gHVcharacterization} for $\alpha>0$ and $\beta>0$. The proof of this case requires much more efforts and is given in several steps in the next section.
\end{remark}

\section{Proof of Theorem \ref{gHVcharacterization}  for $\alpha>0$ and $\beta>0$}
\subsection{Independence  through Kummer transforms}  
We first note that the assumed independence properties imply  equality \eqref{KLeq} for $(s,t,z)\in(0,\infty)\times\R\times(0,\infty)$. Differentiating \eqref{KLeq} with respect to $z$ and dividing side-wise resulting equality by \eqref{KLeq} we get
\begin{equation}\label{4st}
\tfrac{L_X^{(\alpha)}(s+1,t,z)}{L_X^{(\alpha)}(s,t,z)}+\tfrac{L_Y^{(\beta)}(t+1,s,z)}{L_Y^{(\beta)}(t,s,z)}=\tfrac{L_U^{(\alpha)}(t+1,s,z)}{L_U^{(\alpha)}(t,s,z)}+\tfrac{L_V^{(\beta)}(s+1,t,z)}{L_Y^{(\beta)}(s,t,z)}.
\end{equation}
Using identity \eqref{id1} we obtain
$$
\beta\,\tfrac{L_X^{(\alpha)}(s,t-1,z)}{L_X^{(\alpha)}(s,t,z)}+\alpha\,\tfrac{L_Y^{(\beta)}(t,s-1,z)}{L_Y^{(\beta)}(t,s,z)}=\beta\,\tfrac{L_U^{(\alpha)}(t,s-1,z)}{L_U^{(\alpha)}(t,s,z)}+\alpha\,\tfrac{L_V^{(\beta)}(s,t-1,z)}{L_Y^{(\beta)}(s,t,z)}.
$$
Changing in the above formula $s$ to $s+1$ and $t$ to $t+1$ we arrive at
\begin{equation}\label{4st'}
\beta\,\tfrac{L_X^{(\alpha)}(s+1,t,z)}{L_X^{(\alpha)}(s+1,t+1,z)}+\alpha\,\tfrac{L_Y^{(\beta)}(t+1,s,z)}{L_Y^{(\beta)}(t+1,s+1,z)}=\beta\,\tfrac{L_U^{(\alpha)}(t+1,s,z)}{L_U^{(\alpha)}(t+1,s+1,z)}+\alpha\,\tfrac{L_V^{(\beta)}(s+1,t,z)}{L_Y^{(\beta)}(s+1,t+1,z)}.
\end{equation}
Subtracting side-wise \eqref{4st} (multiplied by $\alpha\beta$) from \eqref{4st'}, in view of \eqref{id1}, we get
\begin{equation}\label{4M}
\beta\,M_X^{(\alpha)}(s,t,z)+\alpha\,M_Y^{(\beta)}(t,s,z)=\beta\,M_U^{(\alpha)}(t,s,z)+\alpha\,M_V^{(\beta)}(s,t,z),
\end{equation}
where 
$$
M_W^{(\gamma)}(s,t,z)=\tfrac{L_W^{(\gamma)}(s+1,t,z)L_W^{(\gamma)}(s,t+1,z)}{L_W^{(\gamma)}(s,t,z)L_W^{(\gamma)}(s+1,t+1,z)}.
$$
From now on we suppress superscripts in $L$ and $M$ functions, that is we write $M_X:=M_X^{(\alpha)}$, $L_X:=L_X^{(\alpha)}$, $M_Y:=M_Y^{(\beta)}$, $L_Y:=L_Y^{(\beta)}$, $M_U:=M_U^{(\alpha)}$, $L_U:=L_U^{(\alpha)}$ and  $M_V:=M_V^{(\beta)}$, $L_V:=L_V^{(\beta)}$.   Note  that \eqref{KLeq} implies
\begin{equation}\label{MM}
M_X(s,t,z)M_Y(t,s,z)=M_U(t,s,z)M_V(s,t,z).
\end{equation}

Combining \eqref{4M} with \eqref{MM} we get
\begin{align}
\left(\beta\,M_X(s,t,z)-\alpha\,M_V(s,t,z)\right)\left(M_X(s,t,z)-M_U(t,s,z)\right)&=0\label{double2}\\
\left(\beta\,M_U(t,s,z)-\alpha\,M_Y(t,s,z)\right)\left(M_V(s,t,z)-M_Y(t,s,z)\right)&=0. \label{double1}
\end{align}

Since $M_X,M_Y,M_U, M_V$ all extend uniquely to meromorphic functions in a common domain in $\C^3$ it follows from \eqref{double2} that either $\beta\,M_X\equiv\alpha\,M_V$ or $M_X\equiv M_U$ and  from \eqref{double1} that either  $\beta\,M_U\equiv\alpha\,M_Y$ or $M_V\equiv M_Y$. 

In Section \ref{MX_MV} we will prove that $\beta\,M_X\equiv\alpha\,M_V$ is impossible. It will follow by symmetry that also $\beta\,M_U\equiv\alpha\,M_Y$ is impossible.  
Then, in Section \ref{MXMU} we will consider the case  $M_X\equiv M_U$. It relates $L_X$ and $L_U$ through auxiliary functions $a,b,f$ which are deciphered in Section \ref{abf}. In Section \ref{ode} we eliminate $L_X$ from the problem. It leads to the Kummer differential equation for slightly modified version of $L_U$, which allows to identify $L_U$ (and thus also $L_X$), up to parameters. The case $M_V\equiv M_Y$ follow by the analogous approach. The final step in Section \ref{para} lies in identification of relations between parameters of all four transforms: $L_X$, $L_Y$, $L_U$ and $L_V$.

\subsection{The case $\beta M_X\equiv \alpha M_V$ is impossible}\label{MX_MV}
Assume 
\begin{equation}\label{MXMV}
\beta M_X(s,t,z)=\alpha M_V(s,t,z),\quad s,t\ge 0,\;z>0.
\end{equation}
Define
\begin{equation}\label{Astz}
 \widetilde A(s,t,z):=\tfrac{L_X(s+1,t)L_V(s,t)}{L_X(s,t)L_V(s+1,t)}
\end{equation}
and
\begin{equation}\label{Cstz}
\widetilde B(s,t,z):=\tfrac{L_X(s,t)L_V(s,t+1)}{L_X(s,t+1)L_V(s,t)},
\end{equation}
where we suppressed the argument $z$. Note that \eqref{MXMV} implies
$$
\widetilde A(s,t+1,z)=\tfrac{\beta}{\alpha}\,\widetilde A(s,t,z)\quad \mbox{and}\quad \widetilde B(s+1,t,z)=\tfrac{\alpha}{\beta}\,\widetilde B(s,t,z),\quad s,t\in\N.
$$
Consequently,
$$
\widetilde A(s,t,z)=\left(\tfrac{\beta}{\alpha}\right)^t\,\widetilde a(s,z),\quad\mbox{and}\quad \widetilde B(s,t,z)=\left(\tfrac{\alpha}{\beta}\right)^s\,\widetilde b(t,z),
$$
where $\widetilde a(s,z)=\widetilde A(s,0,z)$ and $\widetilde b(t,z)=\widetilde B(0,t,z)$.

Note  that \eqref{MXMV} implies also
\begin{equation}\label{ALLCLL}
\beta \widetilde A(s,t,z)\tfrac{L_X(s,t)}{L_X(s+1,t+1)}=\alpha \widetilde B(s,t,z)\tfrac{L_V(s,t)}{L_V(s+1,t+1)}.
\end{equation}
Consequently,
$$
h(s,t,z):=\tfrac{\widetilde a(s,z)}{\widetilde b(t,z)}=\left(\tfrac{\alpha}{\beta}\right)^{t+s+1}\tfrac{L_X(s+1,t+1)L_V(s,t)}{L_X(s,t)L_V(s+1,t+1)}.
$$
Then 
$$
\left(\tfrac{\beta}{\alpha}\right)^{s+t+1}\tfrac{\partial h(s,t,z)}{\partial z}\\
=\tfrac{\mathrm{Num}}{\left[L_X(s,t)L_V(s+1,t+1)\right]^2},$$
where the numerator
\begin{align*}
\mathrm{Num}=
&\left[L_X(s+2,t+1)L_V(s,t)+L_X(s+1,t+1)L_V(s+1,t)\right]L_X(s,t)L_V(s+1,t+1)\\
&-L_X(s+1,t+1)L_V(s,t)\left[L_X(s+1,t)L_V(s+1,t+1)+L_X(s,t)L_V(s+2,t+1)\right]\\
=&L_X(s,t)L_V(s,t)\left[L_X(s+2,t+1)L_V(s+1,t+1)-L_X(s+1,t+1)L_V(s+2,t+1)\right]\\
&+L_X(s,t)L_X(s+1,t+1)L_V(s+1,t)L_V(s+1,t+1)\\&-L_X(s+1,t)L_X(s+1,t+1)L_V(s,t)L_V(s+1,t+1)=I_1+I_2-I_3.
\end{align*}
Note that the last two summands in the above expression can be rewritten with the help of \eqref{id1} as follows:
$$
I_2=\tfrac{1}{\beta}L_X(s,t)L_X(s+1,t+1)L_V(s+1,t)\left[L_V(s,t)-L_V(s,t+1)\right]
$$
and
$$
I_3=\tfrac{1}{\alpha}L_X(s+1,t)\left[L_X(s,t)-L_X(s,t+1)\right]L_V(s,t)L_V(s+1,t+1).
$$
Consequently,
\begin{align*}
&I_1+I_2-I_3=L_X(s,t)L_V(s,t)\left\{\tfrac{1}{\alpha}L_V(s+1,t+1)\left[\alpha L_X(s+2,t+1)-L_X(s+1,t)\right]\right.\\&\left.
-\tfrac{1}{\beta}L_X(s+1,t+1)\left[\beta L_V(s+2,t+1)-L_V(s+1,t)\right]\right\}\\
&+\tfrac{1}{\alpha}L_X(s+1,t)L_X(s,t+1)L_V(s,t)L_V(s+1,t+1)\\
&-\tfrac{1}{\beta}L_X(s,t)L_X(s+1,t+1)L_V(s+1,t)L_V(s,t+1).
\end{align*}
Note that \eqref{MXMV} implies that the two last terms cancel. Therefore referring again to \eqref{id1} in the first two expressions above we get
$$
\mathrm{Num}=L_X(s,t)L_V(s,t)L_X(s+1,t+1)L_V(s+1,t+1)\,\tfrac{\alpha-\beta}{\alpha\beta}
$$
whence
$$
\tfrac{\partial\,h(s,t,z)}{\partial\,z}=h(s,t,z)\tfrac{\alpha-\beta}{\alpha\beta},
$$
i.e. 
$$
h(s,t,z)=\chi(s,t)\exp\left(\tfrac{\alpha-\beta}{\alpha\beta}\,z\right),
$$
where $\chi(s,t)$ does not depend on $z$. Consequently,
$$\tfrac{\widetilde  a(s,z)}{\widetilde b(t,z)}=\tfrac{\widetilde a(s)}{\widetilde b(t)}e^{z\tfrac{\alpha-\beta}{\alpha\beta}},$$
where $\widetilde a(s)=\widetilde a(s,0)$ and $\widetilde b(t)=\widetilde b(0,t)$. 

Therefore, referring to \eqref{ALLCLL}, for $s=t=0$ we obtain
$$
e^{-\left(\tfrac{1}{\beta}-\tfrac{1}{\alpha}\right)z}\int_{(0,\infty)^2}\,\tfrac{x}{1+\alpha x}\,e^{-z(x+v)}\P_X(dx)\P_V(dv)=\tfrac{\beta \widetilde a(0)}{\alpha \widetilde b(0)}\,\int_{(0,\infty)^2}\,\tfrac{v}{1+\beta v}\,e^{-z(x+v)}\,\P_X(dx)\P_V(dv)
$$
which, with $\eta= \tfrac{1}{\beta}-\tfrac{1}{\alpha}$,  can be written as
\begin{equation}\label{suppo}
\int_{\R \times (0,\infty)^2}\,\tfrac{x}{1+\alpha x}\,e^{-z(w+x+v)} \delta_{\eta}(dw)\P_X(dx)\P_V(dv)=\tfrac{\beta \widetilde a(0)}{\alpha \widetilde b(0)}\,\int_{(0,\infty)^2}\,\tfrac{v}{1+\beta v}\,e^{-z(x+v)}\,\P_X(dx)\P_V(dv).
\end{equation}
In view of  \eqref{suppo} we see that the support of the convolution $\delta_{\eta} * \P_X * \P_V$ (determined by the left-hand side of  \eqref{suppo})  coincides with the support of the convolution $\P_X * \P_V$ (determined by the right-hand side of  \eqref{suppo}).  Consequently, denoting the latter support by $\mathfrak W$ and letting $w_0=\inf\mathfrak W$, we get  
$\eta+w_0=w_0$, which is impossible since by assumption $\alpha\neq \beta$.

\subsection{The case of $M_X\equiv M_U$ and functions $a$, $b$ and $f$}\label{MXMU}
We consider the equation
\begin{equation}\label{X=U}
M_X(s,t,z)=M_U(t,s,z),\quad s,t\in\{0,1,\ldots\},\; z> 0.
\end{equation}
Denote 
\begin{equation}\label{Astz}
A(s,t,z):=\tfrac{L_X(s+1,t)L_U(t,s)}{L_X(s,t)L_U(t,s+1)},
\end{equation}
and
\begin{equation}\label{Btsz}
B(t,s,z):=\tfrac{L_U(t+1,s)L_X(s,t)}{L_U(t,s)L_X(s,t+1)},
\end{equation}
where we skipped the superscript $(\alpha)$ and the argument $z$ in $L_X$ and $L_U$. 

Note that \eqref{X=U} implies that for all $s,t\in\N=\{0,1,\ldots\}$ we have 
$$
A(s,t,z)=A(s,t+1,z)\quad \mbox{and}\quad B(t,s,z)=B(t,s+1,z).
$$
Consequently, for $(s,t)\in\N^2$ we have $A(s,t,z)=A(s,0,z)=:A(s,z)$ and $B(t,s,z)=B(t,0,z)=:B(t,z)$.
Now \eqref{X=U} can be written as
\begin{equation}\label{AB1}
A(s,z)\,\tfrac{L_X(s,t)}{L_X(s+1,t+1)}=B(t,z)\tfrac{L_U(t,s)}{L_U(t+1,s+1)}.
\end{equation}
Consider now the quotient
$$
\tfrac{A(s,z)}{B(t,z)}=\tfrac{L_X(s+1,t+1)L_U(t,s)}{L_X(s,t)L_U(t+1,s+1)}.
$$ 
Then the numerator $\mathrm{Num}$ of the derivative 
$$\tfrac{\partial\,\tfrac{A(s,z)}{B(t,z)}}{\partial z}$$
assumes the form
\begin{align*}
\mathrm{Num}=&[L_X(s+2,t+1)L_U(t,s)+L_X(s+1,t+1)L_U(t+1,s)]L_X(s,t)L_U(t+1,s+1)\\
-&L_X(s+1,t+1)L_U(t,s)[L_X(s+1,t)L_U(t+1,s+1)+L_X(s,t)L_U(t+2,s+1)]\\
=&L_X(s,t)L_U(t,s)[L_X(s+2,t+1)L_U(t+1,s+1)-L_X(s+1,t+1)L_U(t+2,s+1)]\\
+&L_X(s,t)L_X(s+1,t+1)L_U(t+1,s)\tfrac{1}{\alpha}[L_U(t,s)-L_U(t,s+1)]\\
-&L_X(s+1,t)\tfrac{1}{\alpha}[L_X(s,t)-L_X(s,t+1)]L_U(t,s)L_U(t+1,s+1),
\end{align*}
where we twice used \eqref{id1}. Referring again to \eqref{X=U}, after cancellation, we  get
\begin{align*}
\tfrac{\alpha\,\mathrm{Num}}{L_X(s,t)L_U(t,s)}=&L_U(t+1,s+1)\left[\alpha L_X(s+2,t+1)-L_X(s+1,t)\right]\\-&L_X(s+1,t+1)\left[\alpha L_U(t+2,s+1)-L_U(t+1,s)\right].
\end{align*}
Note that \eqref{id1} applied to the expressions in square brackets above gives $-L_X(s+1,t+1)$ for the first square bracket and $-L_U(t+1,s+1)$ for the second. Consequently, $\mathrm{Num}=0$ and thus $\tfrac{A(s,z)}{B(t,z)}=\tfrac{a(s)}{b(t)}$, where $a(s):=A(s,0)$ and $b(t):=B(t,0)$. Consequently, we have the representations:
\begin{equation}\label{AB}
A(s,z)=f(z)a(s)\quad\mbox{and}\quad B(t,z)=f(z)b(t),\quad z> 0,\;s,t\in\N,
\end{equation}
where $f=\tfrac{A(0,z)}{a(0)}=\tfrac{B(0,z)}{b(0)}$.

Note that \eqref{AB1} can be rewritten as
\begin{equation}\label{AB2}
a(s)\,\tfrac{L_X(s,t,z)}{L_X(s+1,t+1,z)}=b(t)\tfrac{L_U(t,s,z)}{L_U(t+1,s+1,z)}.
\end{equation}

\subsection{Computing $a$, $b$ and $f$}\label{abf}

Taking logarithms of \eqref{Astz} sidewise, using the first equality of \eqref{AB} and differentiating with respect to $z$, in view of \eqref{id2}, we obtain
$$
\tfrac{f'}{f}=-\tfrac{L_X(s+2,t)}{L_X(s+1,t)}+\tfrac{L_X(s+1,t)}{L_X(s,t)}-\tfrac{L_U(t+1,s)}{L_U(t,s)}+\tfrac{L_U(t+1,s+1)}{L_U(t,s+1)}.
$$
Note that 
$$
\tfrac{\partial^2}{\partial z^2}\log\,L_X(s,t)=\tfrac{L_X(s+2,t)}{L_X(s,t)}-\left(\tfrac{L_X(s+1,t)}{L_X(s,t)}\right)^2.
$$
Using the above formula, the identity \eqref{id2} and recalling the definition of $M_U$ we finally get
$$
\tfrac{f'}{f}\,\tfrac{\partial}{\partial z}\log\,L_X(s,t)=\tfrac{\partial^2}{\partial z^2}\log\,L_X(s,t)+\tfrac{L_X(s+1,t)L_U(t+1,s)}{L_X(s,t)L_U(t,s)}\,(1-M_U^{-1}(t,s,z)).
$$
Starting with \eqref{Btsz}, in a similar way, we obtain the analogue of the above
$$
\tfrac{f'}{f}\,\tfrac{\partial}{\partial z}\log\,L_U(t,s)=\tfrac{\partial^2}{\partial z^2}\log\,L_U(t,s)+\tfrac{L_U(t+1,s)L_X(s+1,t)}{L_U(t,s)L_X(s,t)}\,(1-M_X^{-1}(s,t,z)).
$$
Subtracting the last two equalities sidewise, in view of \eqref{X=U}, we obtain
$$
\tfrac{f'}{f}=\tfrac{g'}{g},\qquad \mbox{where}\quad g=\tfrac{\partial}{\partial z}\log\,\tfrac{L_U(t,s)}{L_X(s,t)}.
$$

Consequently,
\begin{equation}\label{Kst}
K(s,t)f=\tfrac{\partial}{\partial z}\log\,\tfrac{L_U(t,s)}{L_X(s,t)}=\tfrac{L_X(s+1,t)}{L_X(s,t)}-\tfrac{L_U(t+1,s)}{L_U(t,s)}
\end{equation}
for some function $K$ which does not depend on $z$. Referring to \eqref{AB} as well as to \eqref{Astz} and \eqref{Btsz} again we get
$$
K(s,t)=a(s)\tfrac{L_U(t,s+1)}{L_U(t,s)}-b(t)\tfrac{L_X(s,t+1)}{L_X(s,t)}.
$$
Now, \eqref{id1} applied to $L_U(t,s+1)$ and $L_X(s,t+1)$ gives
$$
K(s,t)=a(s)-b(t)-\alpha\left(a(s)\tfrac{L_U(t+1,s+1)}{L_U(t,s)}-b(t)\tfrac{L_X(s+1,t+1)}{L_X(s,t)}\right).
$$
Referring to \eqref{AB2} we see that the expression in parenthesis above is zero, whence $K(s,t)=a(s)-b(t)$.

Now we write \eqref{Kst} for $s+1$ and $t$, which gives
$$
K(s+1,t)f=\tfrac{\partial}{\partial z}\log\,\tfrac{L_U(t,s+1)}{L_X(s+1,t)}.
$$
Substracting sidewise this equality from \eqref{Kst} we get
$$
(a(s)-a(s+1))f=\tfrac{\partial}{\partial z}\log\,\tfrac{L_X(s+1,t)L_U(t,s)}{L_X(s,t)L_U(t,s+1)}=\tfrac{\partial}{\partial z}\log\,A(s,z)=\tfrac{d\,\log\,f}{d z}
$$
where the last equality follows from the first part of \eqref{AB}. 

Similarly, using \eqref{Kst} with $s$ and $t+1$ we get 
$$
(b(t)-b(t+1))f=\tfrac{\partial}{\partial z}\log\,B(t,z)=\tfrac{d\,\log\,f}{d z},
$$
where the last equality follows from the second part of \eqref{AB}.

Consequently $a(s)-a(s+1)=b(t)-b(t+1)=:-\kappa\in\R$ and $a(s)=\kappa s +\tilde{a}$ and $b(t)=\kappa t+\tilde{b}$ for some constants  $\tilde a,\,\tilde b\in\R$. Thus
$$
-\kappa f=\tfrac{d\,\log\,f}{d z}.
$$
whence 
\begin{equation}\label{fz}
	f(z)=(\kappa z+C)^{-1}, \quad\mbox{where }\;C\in\R\;\mbox{ is a constant}.
	\end{equation} 

In case $\kappa\neq 0$ we have
$$
A(s,z)=\tfrac{a+s}{c+z}\quad\mbox{and}\quad B(t,z)=\tfrac{b+t}{c+z},
$$
where $a:=\tilde{a}/\kappa$, $b:=\tilde{b}/\kappa$ and $c=C/\kappa$. Since $A(s,z)$ and $B(t,z)$ are strictly positive for all $s\ge 0$, $t\ge 0$ and $z>0$, we conclude that $a,b>0$ and $c\ge 0$.  Note that $\kappa=0$ implies $f\equiv const$, which in view of \eqref{AB} would imply $A(s,z)=const$ and $B(t,z)=const$
.
\subsection{The Kummer ode and  identification (up to parameters) of $L_X$, $L_U$, $L_Y$ and $L_V$}\label{ode}
Note that \eqref{Kst}, when $\kappa\neq 0$, in view of \eqref{fz}, can be rewritten as
$$
\tfrac{K(s,t)}{z+c}=\tfrac{\partial}{\partial z}\,\log\,\tfrac{L_U(t,s)}{L_X(s,t)},\quad z>0,
$$
where we changed $K/\kappa$ into $K$ and $c=C/\kappa$. (Note that with such a new $K(s,t)$ we also have $a(s)=s+a$ and $b(t)=t+b$.) Consequently,
$$
(z+c)^{K(s,t)}=c(s,t)\tfrac{L_U(t,s)}{L_X(s,t)},
$$
where $c(s,t)$ does not depend on $z$. 

Rewrite the above as 
\begin{equation}\label{tozst}
(z+c)^{K(s,t)}L_X(s,t)=c(s,t)L_U(t,s)
\end{equation}
and for $s+1$ and $t$ as 
\begin{equation}\label{tozst1}
(z+c)^{K(s+1,t)}L_X(s+1,t)=c(s+1,t)L_U(t,s+1).
\end{equation}
Since $K(s+1,t)-K(s,t)=1$, dividing \eqref{tozst1} by \eqref{tozst} and referring to \eqref{Astz} we get
\begin{equation}\label{cst}
\tfrac{c(s+1,t)}{c(s,t)}=\tfrac{L_X(s+1,t)L_U(t,s)}{L_X(s,t)L_U(t,s+1)}(z+c)=a(s).
\end{equation}

Now we differentiate \eqref{tozst} with respect to $z$ and get
\begin{equation}\label{after}
K(s,t)(z+c)^{K(s,t)-1}L_X(s,t)-(z+c)^{K(s,t)}L_X(s+1,t)=-c(s,t)L_U(t+1,s).
\end{equation}
For $L_X(s,t)$ and $L_X(s+1,t)$ in \eqref{after} insert the relevant expression from \eqref{tozst} and \eqref{tozst1}, respectively. Together with \eqref{cst} it yields
$$
a(s)L_U(t,s+1)-K(s,t)L_U(t,s)=(c+z)L_U(t+1,s).
$$
Regarding $L_U(t,s)$ and $L_U(t+1,s)$ as the right-hand side of identity \eqref{id1} with suitable $s$ and $t$ we get
$$
\alpha(c+z)L_U(t+2,s+1)+(c+z+\alpha K(s,t))L_U(t+1,s+1)-b(t)L_U(t,s+1)=0.
$$
In view of \eqref{id2}, the above equation transforms into the second order diferential equation for the function $h:=L_U(t,s+1)$ as follows
$$
\alpha(c+z)\,h''(z)+\alpha (b(t)-a(s))-(c+z))\,h'(z)-b(t)\,h(z)=0.
$$
Consequently, for $g$ defined by $g(z)=h(\alpha z-c)$ we get the Kummer equation
\begin{equation}\label{Kum_eq}
zg''(z)+(b(t)-a(s)-z)g'(z)-b(t)g(z)=0.
\end{equation} 
It is well known that the general solution is of the form
$$
g(z)=c_1M(b(t),b(t)-a(s),z)+c_2U(b(t),b(t)-a(s),z),
$$
see 13.1.1, 13.1.2 and 13.1.3 in \cite{AbraSt}). Recall that $M(a,b,z)$ is unbounded when $z\to\infty$ (see e.g. 13.1.4 in \cite{AbraSt}) and $U(a,b,z)\to 0$ as $z\to \infty$. Since $g$, as a Laplace transform of a probability measure, is bounded, we necessarily have
$$
g(z)=c_U(s,t)U(b(t),b(t)-a(s),z).
$$
Returning to $L_U(t,s)$ (recall that $g$ was defined through $L_U(t,s+1)$) we get
$$
L_U(t,s,z)=c_U(s,t)U(b+t,b+t-a-s+1,\tfrac{c+z}{\alpha}),
$$
with $a,b>0$ and $c\ge 0$.

Changing the roles of $L_X$ and $L_U$ in the above argument starting with \eqref{tozst} we obtain
$$
L_X(s,t,z)=c_X(s,t)U(a+s,a+s-b-t+1,\tfrac{c+z}{\alpha}).
$$

Assume that $c=0$. Recalling \eqref{Kufu} we see that: (1)  if $a\neq b$ then either $U(a,a-b+1,0)=\infty$ or $U(b,b-a+1,0)=\infty$; (2) if $a=b$ then $U(a,1,0)=U(b,1,0)=\infty$. Since $L_X(0,0,0)=L_U(0,0,0)=1$ we obtain thus a contradiction. Therefore $c>0$ and Proposition \ref{unique} implies that $X\sim \mathcal K_\alpha(a,b,c)$ and $U\sim\mathcal K_{\alpha}(b,a,c)$. 

In case $\kappa=0$ we have $f(z)=C\neq 0$ and $A(s,z)=a>0$ and $B(t,z)=b>0$ where $a=\tilde{a}{C}$ and $b=\tilde b C$. We now show that this is impossible. 
Indeed, \eqref{Kst} then yields
$$
a-b=\tfrac{L_X(s+1,t)}{L_X(s,t)}-\tfrac{L_U(t+1,s)}{L_U(t,s)}.
$$
Combining this with \eqref{Btsz} we get
$$
(a-b)L_X(s,t)=L_X(s+1,t)-bL_X(s,t+1).
$$
For $s=t=0$ we thus get
$$
\E((X-a+b)e^{-zX}=b\E\,\tfrac{1}{1+\alpha X}\,e^{-zX}.
$$
Consequently, $(x-a+b)\P_X(dx)=\tfrac{b}{1+\alpha x}\,\P_X(dx)$. Equivalently,
$$
\tfrac{(1+\alpha x)(x-a+b)}{b}\,\P_X(dx)=\P_X(dx).
$$
Since $(1+\alpha x)(x-a+b)=b$ is equivalent to $\alpha x^2+(\alpha(b-a)+1)x-a=0$, the roots of which have different signs. Since $X$ is nonnegative this would imply that its support degenerates to a point, which contradicts our assumptions.

\subsection{Identifying the parameters}\label{para}
We have proved that $X\sim \mathcal K_\alpha(a,b,c)$, $U\sim\mathcal K_{\alpha}(b,a,c)$, $Y\sim \mathcal K_\beta(\tilde{a},\tilde{b},\tilde{c})$ and $V\sim\mathcal K_{\beta}(\tilde{b},\tilde{a},\tilde{c})$ for some $a, b, c, \tilde{a},\tilde{b},\tilde{c} >0$. 
Using \eqref{KLTK} for each of the variables $X,Y,U,V$, equation  \eqref{KLeq} reads:
$$\tfrac{\Gamma(a+s)U\left(a+s,a+s-b-t+1,\tfrac{c+z}{\alpha}\right)}{\alpha^s\Gamma(a)U\left(a,a-b+1,\tfrac{c}{\alpha}\right)} \times
\tfrac{\Gamma(\tilde{a}+t)U\left(\tilde{a}+t,\tilde{a}+t-\tilde{b}-s+1,\tfrac{\tilde{c}+z}{\beta}\right)}{\beta^t\Gamma(\tilde{a})U\left(\tilde{a},\tilde{a}-\tilde{b}+1,\tfrac{\tilde{c}}{\beta}\right)}  $$
$$=
\tfrac{\Gamma(b+t)U\left(b+t,b+t-a-s+1,\tfrac{c+z}{\alpha}\right)}{\alpha^t\Gamma(b)U\left(b,b-a+1,\tfrac{c}{\alpha}\right)} \times
\tfrac{\Gamma(\tilde{b}+s)U\left(\tilde{b}+s,\tilde{b}+s-\tilde{a}-t+1,\tfrac{\tilde{c}+z}{\beta}\right)}{\beta^s\Gamma(\tilde{b})U\left(\tilde{b},\tilde{b}-\tilde{a}+1,\tfrac{\tilde{c}}{\beta}\right)}
 $$
which, by applying identity \eqref{AS:Kum} to the left-hand side, gives
$$  \tfrac{\Gamma(a+s) \left(\tfrac{c+z}{\alpha}\right)^{b+t-a-s} U\left(b+t,b+t-a-s+1,\tfrac{c+z}{\alpha}\right)}{\alpha^s\Gamma(a) \left(\tfrac{c}{\alpha}\right)^{b-a}U\left(b,b-a+1,\tfrac{c}{\alpha}\right)} \times
\tfrac{\Gamma(\tilde{a}+t)  \left(\tfrac{\tilde{c}+z}{\beta}\right)^{\tilde{b}+s-\tilde{a}-t} U\left(\tilde{b}+s,\tilde{b}+s-\tilde{a}-t+1,\tfrac{\tilde{c}+z}{\beta}\right)}{\beta^t\Gamma(\tilde{a})  \left(\tfrac{\tilde{c}}{\beta}\right)^{\tilde{b}-\tilde{a}}U\left(\tilde{b},\tilde{b}-\tilde{a}+1,\tfrac{\tilde{c}}{\beta}\right)}  $$
$$=
\tfrac{\Gamma(b+t)U\left(b+t,b+t-a-s+1,\tfrac{c+z}{\alpha}\right)}{\alpha^t\Gamma(b)U\left(b,b-a+1,\tfrac{c}{\alpha}\right)} \times
\tfrac{\Gamma(\tilde{b}+s)U\left(\tilde{b}+s,\tilde{b}+s-\tilde{a}-t+1,\tfrac{\tilde{c}+z}{\beta}\right)}{\beta^s\Gamma(\tilde{b})U\left(\tilde{b},\tilde{b}-\tilde{a}+1,\tfrac{\tilde{c}}{\beta}\right)}.
 $$
After cancellations we obtain
 \begin{equation} \label{ident} \tfrac{   (c+z)^{b+t-a-s}  (\tilde{c}+z)^{\tilde{b}+s-\tilde{a}-t} }
{ c^{b-a}  {\tilde{c}}^{\tilde{b}-\tilde{a} } } \times \tfrac{\Gamma(a+s)}{\Gamma(a)} \times  \tfrac{\Gamma(\tilde{a}+t)}{\Gamma(\tilde{a})} = 
\tfrac{\Gamma(b+t)}{\Gamma(b)} \times  \tfrac{\Gamma(\tilde{b}+s)}{\Gamma(\tilde{b})}. \end{equation} 
Taking the logarithm and differentiating in $z$ gives
$$ \tfrac{b+t-a-s}{c+z} +\tfrac {\tilde{b}+s-\tilde{a}-t}{\tilde{c}+z} =0.$$ 
Since this holds for any $z>0$, we conclude that $c=\tilde{c}$ and $b-a=\tilde{a}-\tilde{b}$.
Returning to \eqref{ident} we have
$$\tfrac{\Gamma(a+s)}{\Gamma(a)} \times  \tfrac{\Gamma(\tilde{a}+t)}{\Gamma(\tilde{a})} = 
\tfrac{\Gamma(b+t)}{\Gamma(b)} \times  \tfrac{\Gamma(\tilde{b}+s)}{\Gamma(\tilde{b})},$$
which for $(s,t)=(0,1)$, yields $\tilde{a} = b$ and for $(s,t)=(1,0)$ yields $a=\tilde b$.

{\bf Acknowledgement:} J. Weso\l owski was supported by grant Beyond POB II no. 1820/366/Z01/2021 within the Excellence Initiative: Research University (IDUB) programme of the Warsaw Univ. of Technology, Poland.


\end{document}